\documentclass[conference]{IEEEtran}
\IEEEoverridecommandlockouts
\usepackage{cite}
\usepackage{amsmath,amssymb,amsfonts}
\usepackage{algorithmic}
\usepackage{graphicx}
\usepackage{textcomp}
\usepackage{xcolor}
\def\BibTeX{{\rm B\kern-.05em{\sc i\kern-.025em b}\kern-.08em
    T\kern-.1667em\lower.7ex\hbox{E}\kern-.125emX}}
    
\DeclareMathOperator*{\argmax}{argmax}

\usepackage{float}
\usepackage{amssymb}

\newcommand{\Matlab}{\textsc{Matlab }}
\newcommand{\mlin}[1]{\mbox{\tt{#1}}}

\DeclareMathOperator{\rank}{rank}
\renewcommand{\vec}{\mathrm{vec}}
\let\tempone\itemize
\let\temptwo\enditemize
\renewenvironment{itemize}{\tempone\addtolength{\itemsep}{0.5\baselineskip}}{\temptwo}
\usepackage[english]{babel}

\usepackage{hyperref}
\usepackage{amsthm}
\usepackage{cleveref}

\newtheorem{theorem}{Theorem}
\newtheorem{lemma}[theorem]{Lemma}
\newtheorem{proposition}[theorem]{Proposition}

\newcommand{\commC}[1]{{ #1}}
\usepackage{enumitem}
\newcommand{\Stiefel}{\mathrm{St}}

\begin{document}

\title{Rank Estimation for Third-Order Tensor Completion in the Tensor-Train Format
}

\author{\IEEEauthorblockN{Charlotte Vermeylen}
\IEEEauthorblockA{\textit{Dept. of Computer Science} \\
\textit{KU Leuven}\\
Heverlee, Belgium \\
charlotte.vermeylen@kuleuven.be}
\and
\IEEEauthorblockN{Guillaume Olikier}
\IEEEauthorblockA{\textit{ICTEAM Institute} \\
\textit{UCLouvain}\\
Louvain-la-Neuve, Belgium \\
guillaume.olikier@uclouvain.be}
\and
\IEEEauthorblockN{P.-A. Absil}
\IEEEauthorblockA{\textit{ICTEAM Institute} \\
\textit{UCLouvain}\\
Louvain-la-Neuve, Belgium \\
pa.absil@uclouvain.be}
\and
\IEEEauthorblockN{Marc Van Barel}
\IEEEauthorblockA{\textit{Dept. of Computer Science} \\
\textit{KU Leuven}\\
Heverlee, Belgium \\
marc.vanbarel@kuleuven.be}
}

\maketitle

\begin{abstract}
We propose a numerical method to obtain an adequate value for the upper bound on the rank for the tensor completion problem on the variety of third-order tensors of bounded tensor-train rank. The method is inspired by the parametrization of the tangent cone derived by Kutschan (2018). A proof of the adequacy of the upper bound for a related low-rank tensor approximation problem is given and an estimated rank is defined to extend the result to the low-rank tensor completion problem. Some experiments on synthetic data illustrate the approach and show that the method is very robust, e.g., to noise on the data.  
\end{abstract}


\section{Introduction}
We consider the \emph{low-rank tensor completion problem (LRTCP)} formulated as a least squares optimization problem on the algebraic variety $\mathbb{R}_{\le (k_1, k_2)}^{n_1 \times n_2 \times n_3}$ \cite[Definition 1.4]
{kutschan2018} of $n_1 \times n_2 \times n_3$ real third-order tensors of tensor-train (TT) rank at most $(k_1,k_2)$: 
\begin{equation} \label{eq:min_completion_TT}
\min_{X \in \mathbb{R}_{\le (k_1, k_2)}^{n_1 \times n_2 \times n_3}} \underbrace{\frac{1}{2} \lVert X_{\Omega} - A_{\Omega} \rVert^2}_{=:f_\Omega(X)},
\end{equation}
where $A \in \mathbb{R}^{n_1 \times n_2 \times n_3}$, $\Omega \subseteq \lbrace 1, \dots, n_1 \rbrace \times \lbrace 1, \dots, n_2 \rbrace \times \lbrace 1, \dots, n_3 \rbrace$ is called the sampling set,
\begin{equation*} 
Z_{\Omega}\left(i_1,i_2,i_3\right) := \begin{cases}
Z\left(i_1,i_2,i_3 \right)  & \text{if }\left(i_1,i_2,i_3 \right) \in \Omega, \\
0 & \text{otherwise}, \\
\end{cases}
\end{equation*}
for all $Z \in \mathbb{R}^{n_1 \times n_2 \times n_3}$, and the norm is induced by the inner product
\begin{equation} \label{eq:inner_product}
\langle Y ,X \rangle = \langle \vec(Y) ,\vec(X) \rangle, \quad \forall X,Y \in \mathbb{R}^{n_1 \times n_2 \times n_3}.
\end{equation}

A \emph{tensor-train decomposition} (TTD) of a third-order tensor $X \in \mathbb{R}^{n_1 \times n_2 \times n_3}$ is a factorization $X = X_1 \cdot X_2 \cdot X_3$, where $X_1 \in \mathbb{R}^{n_1 \times r_1}$, $X_2 \in \mathbb{R}^{r_1 \times n_2 \times r_2}$, and $X_3 \in \mathbb{R}^{r_2 \times n_3}$ \cite{Oseledets2011}. The `$\cdot$' indicates the contraction between a matrix and a tensor. They interact with the left and right unfolding of $X_2$, 
\begin{equation*}
\begin{split}
X_2^\mathrm{L} &:= \left[ X_2 \right]^{r_1 \times n_2 r_2}:= \mathrm{reshape}\left( X_2, r_1 \times n_2 r_2 \right), \\
X_2^\mathrm{R} &:= \left[ X_2 \right]^{r_1 n_2 \times r_2} := \mathrm{reshape}\left( X_2, r_1 n_2 \times r_2 \right),
\end{split}
\end{equation*}
in the following way:
\begin{equation*} 
\begin{split}
X_1 \cdot X_2 = \left[ X_1 X_2^\mathrm{L} \right]^{n_1 \times n_2 \times r_2},  X_2 \cdot X_3= \left[ X_2^\mathrm{R} X_3 \right]^{r_1 \times n_2 \times n_3}.\\
\end{split}
\end{equation*}
The minimal $r_1$ and $r_2$ for which a TTD of $X$ exists, is called the \emph{TT-rank} of $X$. For second-order tensors (matrices), the {TT-rank} reduces to the usual matrix rank, and since no other definition of tensor rank is used in this paper, it is simply denoted by $\rank X$ and can be determined as
\begin{equation} \label{eq:def_rank_TT}
\rank X = \left (\rank X^\mathrm{L}, \rank X^\mathrm{R} \right)=: (r_1,r_2).
\end{equation}
The minimal rank decomposition can be obtained by computing successive singular value decompositions (SVDs) of the unfoldings \cite[Algorithm 1]{Oseledets2011}. 

The low-rank variety can then be defined as
\begin{equation} \label{eq:variety}
\mathbb{R}_{\le (k_1, k_2)}^{n_1 \times n_2 \times n_3} := \{X \in \mathbb{R}^{n_1 \times n_2 \times n_3} \mid \rank X \le (k_1, k_2)\}
\end{equation}
and the fixed-rank smooth manifold \cite{Holtz_manifolds_2012} as 
\begin{equation} \label{eq:fixed_rank_manifold}
\mathbb{R}_{(r_1, r_2)}^{n_1 \times n_2 \times n_3} := \{X \in \mathbb{R}^{n_1 \times n_2 \times n_3} \mid \rank X = (r_1, r_2)\}.
\end{equation}

In practical LRTCPs, the rank of $A$ is not known or $A$ has full rank due to noise. This is notably the case for movie rating recommendation systems \cite{Movielens_summary}, where, e.g., the ratings of different users over time (or any other variable of interest) form samples of a large third-order tensor. Evaluating a solution to \eqref{eq:min_completion_TT} in elements outside the set $\Omega$ allows us to recommend movies with a high estimated rating. Note that evaluating one element of a third-order TTD corresponds to performing a vector-matrix-vector multiplication and can be done efficiently in $\mathcal{O}\left( r_1r_2\right)$ operations. 

When $k_1$ and $k_2$ are set too high however, the complexity of an algorithm to solve \eqref{eq:min_completion_TT} is unnecessarily high and furthermore overfitting can occur, i.e., $X$ approximates $A_\Omega$ well but not the full tensor $A$. To detect overfitting, usually a test data set $\Gamma$ is used \cite{Steinl_high_dim_TT_compl_2016}. When the error on this test set increases during optimization while the error of \eqref{eq:min_completion_TT} decreases overfitting has occurred and the algorithm should be stopped or the rank decreased. On the other hand, when $k_1$ or $k_2$ are set too low, the search space may not contain a sufficiently good approximation of $A$. It is thus important to choose adequate values for $k_1$ and $k_2$. 

Intuitively the smaller $\lvert \Omega \rvert$, the more difficult it is to recover $A$ from $A_{\Omega}$ by solving \eqref{eq:min_completion_TT}. However, the minimal number of samples needed is not known \cite{budzinskiy2021tensor}.

In this work, a method to estimate the rank of $A$ from $A_\Omega$ is proposed. When $A$ is not exactly low rank, a good value for a low-rank approximation is obtained. This method can then be used, e.g., in a rank-adaptive optimization algorithm to solve \eqref{eq:min_completion_TT}. 

\commC{The paper is organized as follows. First, in \cref{sec:TT}, the preliminaries for \cref{sec:tangent_cone} and \cref{sec:rank_est} are given. This includes some basic facts concerning orthogonal projections onto vector spaces and arithmetic rules for TTDs. For a more extensive overview of properties of the TTD, we refer to the original paper \cite{Oseledets2011} and the notation introduced in \cite{LR_tensor_methods_Steinlechner_2014}, which is also used in \cite{Steinl_high_dim_TT_compl_2016,Steinlechner_thesis2016}. In \cref{sec:low_rank_approx}, the auxiliary low-rank tensor approximation problem (LRTAP) is defined. 
In \cref{sec:tangent_cone}, the parametrization of the tangent cone to the low-rank variety is given \cite{kutschan2018}. New orthogonality conditions are derived to ensure that in \Cref{prop:rank_est_best_low_rank_approx} no matrix inverse is needed, which improves the stability of the proposed method, and makes the proofs in the rest of the paper easier. In \cref{sec:rank_est}, the main proposition is derived, and an estimated rank is defined to extend our result to the LRTCP. Lastly, in \cref{sec:experiments}, some experiments illustrate the use and advantages of the proposed rank estimation method for the LRTCP.}

\section{Preliminaries} \label{sec:TT}
A TTD is not unique.
Orthogonality conditions can be enforced, which can improve the stability of algorithms working with TTDs. Those used in this paper are introduced in \cref{sec:tangent_cone}.

Given $n, p \in \mathbb{N}$ with $n \ge p$, we let $\Stiefel(p, n) := \{U \in \mathbb{R}^{n \times p} \mid U^\top U = I_p\}$ denote the Stiefel manifold. For every $U \in \Stiefel(p, n)$, we let $P_U := UU^\top$ and $P_U^\perp := I_n-P_U$ denote the orthogonal projections onto the range of $U$ and its orthogonal complement, respectively.
A tensor is said to be \emph{left-orthogonal} if $n_1 \leq n_2n_3$ and $\left( X^\mathrm{L} \right)^\top \in \Stiefel(n_1,n_2n_3)$, and \emph{right-orthogonal} if $n_3 \leq n_1 n_2$ and $ X^\mathrm{R} \in \Stiefel(n_3,n_1n_2)$.

The following properties and arithmetic rules are used frequently in the rest of the paper.
\begin{itemize}
    \item For all matrices $Y$ and $Z$, it holds that $Y \cdot (X_1 \cdot {X}_2 \cdot X_3) \cdot Z = YX_1 \cdot {X}_2 \cdot X_3 Z$. 
    \item The left and right unfoldings of $X=X_1 \cdot X_2 \cdot X_3$ can be rewritten as:
    \begin{equation} \label{eq:LR_TTD}
\begin{split}
&X^\mathrm{L} = X_1 \left(  X_2 \cdot X_3 \right)^\mathrm{L} = X_1 X_2^\mathrm{L} \left(X_3 \otimes I_{n_2} \right),\\
&X^\mathrm{R} =\left( X_1 \cdot X_2  \right)^\mathrm{R} X_3 = \left( I_{n_2} \otimes X_1 \right)  X_2^\mathrm{R} X_3,
\end{split}
\end{equation}
where `$\otimes$' denotes the Kronecker product. 
\item From \eqref{eq:def_rank_TT} and \eqref{eq:LR_TTD} it can be deduced that
\begin{equation} \label{eq:rank_factors_TTD1}
\begin{split}
\rank (X_1) &= \rank \left( X_2^\mathrm{L} \left(X_3 \otimes I_{n_2} \right) \right) = r_1, \\
\rank (X_3) &= \rank \left(\left( I_{n_2} \otimes X_1 \right)  X_2^\mathrm{R} \right) = r_2, \\
\end{split}
\end{equation}
and because the ranks of $I_{n_2} \otimes X_1$ and $X_3 \otimes I_{n_2}$ are $n_2 r_1$ and $r_2 n_2$, respectively, \eqref{eq:rank_factors_TTD1} can be simplified to
$\rank \left( X_2^\mathrm{L} \right) = r_1$ and $\rank \left( X_2^\mathrm{R} \right) = r_2.$
\item Orthogonality between TTDs is exploited frequently in the parametrization of the tangent cone in \cref{sec:tangent_cone} and the proofs in \cref{sec:rank_est}. If $Y = Y_1 \cdot Y_2 \cdot Y_3$ and $Z = Z_1 \cdot Z_2 \cdot Z_3$, then by using \eqref{eq:LR_TTD}, the inner product $\langle Y,Z \rangle$ is zero if at least one of the following equalities holds:
\begin{align} \label{eq:inner_prod_zero}
Y_1^\top Z_1 &= 0, && &\left( Y_2 \cdot Y_3 \right)^\mathrm{L} \big( \left( Z_2 \cdot Z_3 \right)^\mathrm{L} \big)^\top &= 0, & \nonumber \\
Y_3 Z_3^\top &= 0, && &\big( \left( Y_1 \cdot Y_2 \right)^\mathrm{R} \big)^\top \left( Z_1 \cdot Z_2 \right)^\mathrm{R}  &= 0. &  
\end{align}
\end{itemize} 
\section{Low-Rank Tensor Approximation} \label{sec:low_rank_approx}
The \emph{low-rank tensor approximation problem} (LRTAP) is defined as:
\begin{equation} \label{eq:low-rank_approx}
\min_{X \in \mathbb{R}_{\le (k_1, k_2)}^{n_1 \times n_2  \times n_3}} \underbrace{\frac{1}{2} \lVert X - A  \rVert^2}_{=: f(X)}.
\end{equation}
This problem is related to the LRTCP \eqref{eq:min_completion_TT} because $f_\Omega(X) = f(X)$ for $\Omega = \lbrace 1, \dots, n_1 \rbrace \times \lbrace 1, \dots, n_2 \rbrace \times \lbrace 1, \dots, n_3 \rbrace$. Remark that, as for \eqref{eq:min_completion_TT}, a global minimizer is, in general, not unique because $\mathbb{R}_{\le (k_1, k_2)}^{n_1 \times n_2  \times n_3}$ is nonconvex and NP-hard to obtain \cite{hillar2013most}. This problem is used in \cref{sec:rank_est}.

\section{Tangent Cone} \label{sec:tangent_cone}
The set of all tangent vectors to $\mathbb{R}_{\le (k_1, k_2)}^{n_1 \times n_2 \times n_3}$ at $X = X_1'\cdot X_2' \cdot X_3 \in \mathbb{R}_{(r_1, r_2)}^{n_1 \times n_2 \times n_3}$, where $X_1' \in \Stiefel(r_1,n_1)$, ${X_2'}^\mathrm{R} \in \Stiefel(r_2,r_1n_2)$, and $k_1 \ge r_1$, $k_2 \ge r_2$, is a closed cone called the tangent cone to $\mathbb{R}_{\le (k_1, k_2)}^{n_1 \times n_2 \times n_3}$ at $X$ and denoted by $T_X \mathbb{R}_{\le (k_1, k_2)}^{n_1 \times n_2 \times n_3}$. By \cite[Theorem 2.6]{kutschan2018}, $T_X\mathbb{R}_{\le (k_1, k_2)}^{n_1 \times n_2 \times n_3}$ is the set of all $G \in \mathbb{R}^{n_1 \times n_2 \times n_3}$ that can be decomposed as
\begin{equation*} 
\begin{split}
G &= \begin{bmatrix} X_1' & U_1 & W_1 \end{bmatrix} \cdot
\begin{bmatrix}
X_2' & U_2 & W_2 \\ 
0 & Z_2 & V_2 \\
0 & 0 & X_2' \\
\end{bmatrix} \cdot \begin{bmatrix} W_3 \\ V_3 \\ X_3 \end{bmatrix}, \\
\end{split}
\end{equation*}
where $U_1 \in \mathbb{R}^{n_1 \times s_1}$, $W_1 \in \mathbb{R}^{n_1 \times r_1}$, $U_2 \in \mathbb{R}^{r_1 \times n_2 \times s_2}$, $W_2 \in \mathbb{R}^{r_1 \times n_2 \times r_2}$, $Z_2 \in \mathbb{R}^{s_1 \times n_2 \times s_2}$, $V_2 \in \mathbb{R}^{s_1 \times n_2 \times r_2}$, $W_3 \in \mathbb{R}^{r_2 \times n_3}$, $V_3 \in \mathbb{R}^{s_2 \times n_3}$, and $s_i=k_i - r_i$. As for the TTD itself, this parametrization is not unique. In \cite{kutschan2018}, the following orthogonality conditions are derived:
\begin{align*} 
{U}_1^\top X_1' &= 0, &   {W}_1^\top X_1' &= 0, \nonumber\\
\left({U}_2^{\mathrm{R}} \right)^\top X_2'^{\mathrm{R}} &= 0, & \left( {V}_2 \cdot X_3 \right)^{\mathrm{L}} \left(  \left( X_{2}' \cdot X_3 \right)^\mathrm{L} \right)^\top &= 0, \\
\left({W}_2^{\mathrm{R}} \right)^\top X_2'^{\mathrm{R}}&=0, & {V}_3 X_{3}^\top &= 0.  \nonumber
\end{align*} 
We change these orthogonality conditions slightly to make the proofs in the rest of this paper easier and the computations in the experiments more stable. To do so, we notice that 
\begin{equation*}
\begin{split}
G = &\begin{bmatrix} X_1' & U_1  & \dot{W}_1 \end{bmatrix} \cdot \begin{bmatrix}
X_2' & U_2  &  \dot{W}_2\\ 
0 &  Z_2 & \dot{V}_2 \\
0 & 0 &  X_2'' \\
\end{bmatrix} \cdot \begin{bmatrix} W_3 \\  V_3 \\ X_3'' \end{bmatrix}, \\
\end{split}
\end{equation*}
where we have defined $\dot{W}_1 := W_1 R^{-1}$, $\dot{W}_2 := W_2 \cdot C$, $\dot{V}_2 := V_2 \cdot C$, $X_2'' := R \cdot  X_2' \cdot C$, $X_3'' := C^{-1} X_3$, and $R \in \mathbb{R}^{r_1 \times r_1}$ and $C \in \mathbb{R}^{r_2 \times r_2}$ are chosen such that $X_2''$ and $X_3''$ are left-orthogonal. 
Thus, we can also define $X_1:= X_1' R^{-1}$ and $X_2: = X_2' \cdot C$, such that $X = X_1 \cdot X_2'' \cdot X_3'' = X_1' \cdot X_2 \cdot X_3''$. 
Additionally, $W_3$ can be decomposed as $W_3 =  W_3 X_3''^\top X_3'' +  \hat{W}_3$. The two terms involving $W_3$ and $\dot{W}_2$ can then be regrouped as
\begin{equation*} 
\begin{split}
 X_1' \cdot \left(  X_2' \cdot W_3 X_3''^\top + \dot{W}_2\right) \cdot X_3''+  X_1' \cdot X_2' \cdot \hat{W}_3.
\end{split}
\end{equation*}
Defining $\tilde{W}_2 := X_2' \cdot W_3 X_3''^\top + \dot{W}_2$, we obtain
\begin{equation} \label{eq:paran_g_final}
\begin{split}
G = &\begin{bmatrix} X_1' & U_1  & \dot{W}_1 \end{bmatrix} \cdot \begin{bmatrix}
X_2' & U_2  &  \tilde{W}_2\\ 
0 &  Z_2 & \dot{V}_2 \\
0 & 0 &  X_2'' \\
\end{bmatrix} \cdot \begin{bmatrix} \hat{W}_3 \\  V_3 \\ X_3'' \end{bmatrix}, \\
\end{split}
\end{equation}
with the modified orthogonality conditions
\begin{align} \label{eq:orth_cond_K_modified}
{U}_1^\top X_1' &= 0, &  \dot{W}_1^\top X_1' &= 0, & \left({U}_2^{\mathrm{R}} \right)^\top X_2'^{\mathrm{R}} &= 0, \nonumber\\
\hat{W}_3  X_3''^{\top}&=0, & {V}_3 X_{3}''^\top &= 0, & \dot{V}_2^{\mathrm{L}}  \left( X_{2}''^{\mathrm{L}} \right)^\top &= 0. 
\end{align} 
Expanding \eqref{eq:paran_g_final}, a sum of 6 mutually orthogonal TTDs is obtained because of \eqref{eq:inner_prod_zero} and \eqref{eq:orth_cond_K_modified}.

\commC{The projection onto the closed cone $T_X \mathbb{R}_{\leq (k_1 , k_2)}^{n_1 \times n_2 \times n_3}$ is not known and, in general, difficult to obtain because it is nonlinear and nonconvex. However, in what follows we show that any tensor $Y \in \mathbb{R}^{n_1 \times n_2 \times n_3}$ is an element of $T_X \mathbb{R}_{\leq (n_1, n_3)}^{n_1 \times n_2 \times n_3}$. Straightforward computations show that
\begin{equation} \label{eq:Y_6terms}
\begin{split}
Y = ~&P_{X_1'} \cdot Y \cdot P_{X_3''^\top} + P_{X_1'}^\perp \cdot Y \cdot P_{X_3''^\top} \\ 
+ ~&P_{X_1'} \cdot Y \cdot P_{X_3''^\top}^\perp + P_{X_1'}^\perp \cdot Y \cdot P_{X_3''^\top}^\perp\\
= ~&P_{X_1'} \cdot Y \cdot P_{X_3''^\top} \\
+ ~&\left[ P_{X_1'}^\perp \left(  Y \cdot X_3''^\top \right)^\mathrm{L} P_{\left( X_2''^{\mathrm{L}}\right)^\top} \right]^{n_1 \times n_2 \times r_2} \cdot  X_3''\\ 
 + ~& \left[ P_{X_1'}^\perp \left(  Y \cdot X_3''^\top \right)^\mathrm{L} P_{\left( X_2''^{\mathrm{L}}\right)^\top}^\perp \right]^{n_1 \times n_2 \times r_2} \cdot X_3''\\
+ ~&X_1' \cdot \left[ P_{X_2'^{\mathrm{R}}} \left( X_1'^\top \cdot Y  \right)^{\mathrm{R}} P_{X_3''^\top}^\perp \right]^{r_1 \times n_2 \times n_3} \\
+ ~&X_1' \cdot \Big[ P_{X_2'^{\mathrm{R}}}^\perp \left( X_1'^\top \cdot Y  \right)^{\mathrm{R}} P_{X_3''^\top}^\perp \Big]^{r_1 \times n_2 \times n_3}\\
+ ~&P_{X_1'}^\perp \cdot Y \cdot P_{X_3''^\top}^\perp. \\
\end{split}
\end{equation} 
Thus, $Y=Y_1 \cdot Y_2 \cdot Y_3 \in \mathbb{R}_{\leq (n_1, n_3)}^{n_1 \times n_2 \times n_3}$ can be parameterized as in \eqref{eq:paran_g_final}, with 
\begin{equation} \label{eq:param_Y_TC}
\begin{split}
U_1 &= P_{X_1'}^\perp Y_1, \\
U_2 &= \Big[ P_{X_2'^{\mathrm{R}}}^\perp \left( X_1'^\top \cdot Y_1 \cdot Y_2 \right)^{\mathrm{R}} \Big]^{r_1 \times n_2 \times n_3}, \\
\dot{V}_2 &=  \Big[ \left( Y_2 \cdot Y_3 \cdot X_3''^\top \right)^\mathrm{L} P_{\left( X_2''^{\mathrm{L}}\right)^\top}^\perp \Big]^{n_1 \times n_2 \times r_2},  \\
{V}_3 &= Y_3 P_{X_3''^\top}^\perp, \\
\dot{W}_1& = \left( P_{X_1'}^\perp \cdot Y \cdot X_3''^\top \right)^\mathrm{L}  \left( X_2''^{\mathrm{L}}\right)^\top,\\
\tilde{W}_2 &=  X_1'^{\top} \cdot Y \cdot X_3''^{\top}, \\
\hat{W}_3 &= 
\left( X_2'^{\mathrm{R}}\right)^\top \left( X_1'^\top \cdot Y \cdot P_{X_3''^\top}^\perp \right)^{\mathrm{R}}, \quad \\
Z_2 &= Y_2, \\
\end{split}
\end{equation} 
which satisfies \eqref{eq:orth_cond_K_modified}. Furthermore, $U_1$ and $V_3$ have rank at most $n_1 - r_1 $ and $n_3 - r_2$, respectively, and thus the matrix $U_1 \dot{V}_2^\mathrm{L}$ has rank $n_1 - r_1 $, and $U_1$ and $\dot{V}_2^\mathrm{L}$ can be reduced to size $n_1 \times \left( n_1 - r_1 \right)$ and $\left( n_1 - r_1 \right) \times n_2 r_2$, respectively, e.g., by computing the SVD of $U_1 \dot{V}_2^\mathrm{L}$. Similarly, this can be done for $U_2^\mathrm{R} V_3$ to obtain $U_2^\mathrm{R} $ and $V_3$ of size $r_1 n_2 \times \left( n_3 - r_2 \right)$ and $\left( n_3 - r_2 \right) \times n_3$, respectively. Then, $Z_2$ can be changed accordingly to $U_1^\top \cdot Y_2 \cdot V_3^\top$, which is the same result as would be obtained by the TT-rounding algorithm \cite[Algorithm 2]{Oseledets2011}, except for the orthogonality conditions. Thus, $Y$ can be written in the form \eqref{eq:paran_g_final} with $s_1= n_1 - r_1$, $s_2 = n_3 - r_2$ and hence by definition $Y \in T_X \mathbb{R}_{\leq (n_1, n_3)}^{n_1 \times n_2 \times n_3}$.}

The Riemannian gradient of \eqref{eq:low-rank_approx} at $X \in \mathbb{R}_{(r_1, r_2)}^{n_1 \times n_2 \times n_3}$ is  defined as the projection of the Euclidean gradient $\nabla f(X) = X-A$ onto the tangent space \cite{Holtz_manifolds_2012}:
\begin{equation*} 
\begin{split}
&T_X\mathbb{R}_{(r_1, r_2)}^{n_1 \times n_2 \times n_3} := \\
& \left\lbrace\begin{array}{l}
X_1' \cdot X_2' \cdot \hat{W}_3 + X_1' \cdot \tilde{W}_2 \cdot X_3'' + \dot{W}_1 \cdot X_2'' \cdot X_3''
\end{array}\right\rbrace.
\end{split}
\end{equation*} 
By replacing $Y$ by $X-A$ in \eqref{eq:param_Y_TC}, the parameters of $\mathcal{P}_{T_X\mathbb{R}_{(r_1, r_2)}^{n_1 \times n_2 \times n_3}} \nabla f(X)$ are:
\begin{equation} \label{eq:proj_tangent_space_param}
\begin{split}
\dot{W}_1 &= - \left(  P_{X_1'}^\perp \cdot A \cdot X_3''^\top \right)^\mathrm{L}  \left( X_2''^{\mathrm{L}}\right)^\top,  \\ 
\tilde{W}_2 &= {X}_2 - X_1'^{\top} \cdot A \cdot X_3''^{\top}, \\
\hat{W}_3 &= 
-\left( X_2'^{\mathrm{R}}\right)^\top \left( X_1'^\top \cdot A \cdot P_{X_3''^\top}^\perp \right)^{\mathrm{R}}.
\end{split}
\end{equation}
A similar projection onto the tangent space was used in \cite{Steinl_high_dim_TT_compl_2016} and \cite{Lubich_TT_time_int_2015}, but with different orthogonality conditions. 

\section{Rank Estimation} \label{sec:rank_est}
\Cref{prop:rank_est_best_low_rank_approx} states the main result for the LRTAP \eqref{eq:low-rank_approx}. Afterwards, the estimated rank \eqref{eq:num_rank_increase} is defined to extend this result to the LRTCP \eqref{eq:min_completion_TT}. To prove \Cref{prop:rank_est_best_low_rank_approx}, the following auxiliary lemma is used. 

\begin{lemma} \label{lemma:local_min_best_approx}
Let $X = X_1' \cdot {X}_2 \cdot X_3'' = X_1' \cdot {X}_2' \cdot X_3 = X_1 \cdot {X}_2'' \cdot X_3'' \in \mathbb{R}_{(r_1, r_2)}^{n_1 \times n_2 \times n_3}$, with $X_1' \in \Stiefel(r_1,n_1)$, ${X_2'}^\mathrm{R} \in \Stiefel(r_2,r_1 n_2)$, $\left({X_2''}^\mathrm{L}\right)^\top \in \Stiefel(r_1,n_2 r_2)$, and ${X_3''}^\top \in \Stiefel(r_2,n_3)$, and let $A = A_1' \cdot {A}_2 \cdot A_3'' \in \mathbb{R}_{(r_1', r_2')}^{n_1 \times n_2 \times n_3}$, with $A_1' \in \Stiefel(r_1',n_1)$ and ${A_3''}^\top \in \Stiefel(r_2',n_3)$. If $\mathcal{P}_{T_{X}\mathbb{R}_{(r_1, r_2)}^{n_1 \times n_2 \times n_3}} \nabla f \left(X \right)=0$, then
\begin{align*}
X_1'= A_1' B_1, && X_3'' = B_3 A_3'', && {X}_2 = B_1^\top \cdot {A}_2 \cdot B_3^\top,
\end{align*}
for some $B_1 \in \Stiefel(r_1,r'_1)$ and $B_3^\top \in \Stiefel(r_2,r_2')$.
\end{lemma}
\begin{proof}
From \eqref{eq:proj_tangent_space_param}, $\mathcal{P}_{T_{X}\mathbb{R}_{(r_1,r_2)}^{n_1 \times n_2  \times n_3}}\nabla f \left(X \right)=0$ if and only if $\dot{W}_1 = 0$, $\tilde{W}_2 = 0$, and $\hat{W}_3 = 0$.
From the second equation in \eqref{eq:proj_tangent_space_param}, it is clear that $\tilde{W}_2$ can only be zero if $X_2 = X_1'^{\top} \cdot A \cdot X_3''^{\top}$. The matrices $X_1'$ and $X_3''$ are decomposed as
\begin{align} \label{eq:X1_X3_ifv_A1_A3}
&X_1' = \begin{bmatrix} A_1'  & A_1'^{\perp} \end{bmatrix} \begin{bmatrix} B_1  \\ B_2 \end{bmatrix} 
= A_1' B_1 + A_1'^{\perp} B_2,\\
&X_3'' =\begin{bmatrix} B_3  & B_4 \end{bmatrix} \begin{bmatrix} A_3''  \\ A_3''^{\perp} \end{bmatrix} = B_3 A_3'' + B_4 A_3''^{\perp} \nonumber,
\end{align}
where $\begin{bmatrix} A_1'  & A_1'^{\perp} \end{bmatrix} \in \Stiefel(n_1, n_1)$, $\begin{bmatrix} {A_3''} \\ A_3''^{\perp}  \end{bmatrix} \in \Stiefel(n_3, n_3)$, $\begin{bmatrix} B_1 \\ B_2 \end{bmatrix} \in \Stiefel(r_1, r_1 + r_1')$, and $\begin{bmatrix} B_3^\top \\ B_4^\top \end{bmatrix} \in \Stiefel(r_2, r_2 + r_2')$. Substituting \eqref{eq:X1_X3_ifv_A1_A3} into the equation for $X_2$, we obtain $X_2 = B_1^\top \cdot A_2 \cdot B_3^\top$.

Substituting \eqref{eq:X1_X3_ifv_A1_A3} into the equation for $\dot{W}_1$ in \eqref{eq:proj_tangent_space_param}, we get
\begin{equation*}
\begin{split}
&\dot{W}_1 = 0 \Leftrightarrow
 \Big( \Big( I_{n_1} - A_1' B_1 B_1^{\top} A_1'^{\top} - A_1'^{\perp} B_2 B_1^{\top} A_1'^{\top} -  \\
& A_1' B_1 B_2^{\top} \left( A_1'^\perp \right)^{\top} -A_1'^{\perp} B_2 B_2^{\top} \left( A_1'^\perp \right)^{\top} \Big)A_1' \cdot A_2 \cdot A_3'' \\
& \left( A_3''^{\top} B_3^{\top} + \left(A_3''^\perp \right)^{\top} B_4^{\top} \right) \Big)^\mathrm{L} \left( {X_2''}^{\mathrm{L}}\right)^{\top} = 0 \Leftrightarrow\\ 
&\left( \left( A_1' - A_1' B_1 B_1^{\top} - A_1'^{\perp} B_2 B_1^{\top} \right) \cdot A_2 \cdot B_3^{\top} \right)^{\mathrm{L}} \big( {X_2''}^{\mathrm{L}}\big)^{\top} = 0. \\ 
\end{split}
\end{equation*}
Multiplying both sides by $\left(A_1'^\perp \right)^\top$, we obtain
\begin{equation*}
\begin{split}
\dot{W}_1 = 0 &\Rightarrow
B_2 B_1^{\top} A_2^{\mathrm{L}}  \left( B_3^{\top} \otimes I_{n_2} \right) \left(  {X_2''}^{\mathrm{L}}\right)^{\top} = 0 \\
&\Leftrightarrow
B_2  {X_2}^{\mathrm{L}} \left( {X_2''}^{\mathrm{L}}\right)^{\top} = 0  \\
&\Leftrightarrow
B_2 R^{-1} {X_2''}^{\mathrm{L}} \left( {X_2''}^{\mathrm{L}}\right)^{\top} = 0 \Leftrightarrow
B_2 = 0.
\end{split}
\end{equation*}
Thus, it holds that $X_1' = A_1' B_1$ with $B_1 \in \Stiefel(r_1,r_1')$. From $\hat{W}_3 =0$ in \eqref{eq:proj_tangent_space_param},
it can similarly be derived that $B_4 = 0$. Hence, $X_3'' = B_3 A_3''$ and $B_3^\top \in \Stiefel(r_2,r_2')$. 
\end{proof}

The two equalities in \Cref{prop:rank_est_best_low_rank_approx} enable to deduce the TT-rank of $A$---and thus a value of $(k_1, k_2)$ for which the optimum of LRTAP \eqref{eq:low-rank_approx} is zero---from any stationary point of $\min_{X \in \mathbb{R}_{(r_1, r_2)}^{n_1 \times n_2 \times n_3}} f(X)$.

\begin{proposition}
\label{prop:rank_est_best_low_rank_approx}
If the same conditions as in \Cref{lemma:local_min_best_approx} hold, 
then $\nabla f(X) \in T_{X}\mathbb{R}_{\le (r_1', r_2')}^{n_1 \times n_2 \times n_3}$, and
\begin{equation*}
\begin{split}
&\left( \rank \left( \left( \nabla f(X) \cdot X_3''^\top \right)^\mathrm{L} \right),\rank \left( \left( X_1'^\top \cdot \nabla f(X) \right)^R \right)\right) \\
&= \left( r'_1-r_1, r'_2-r_2 \right).
\end{split}
\end{equation*}
\end{proposition}
\begin{proof}
By decomposing $\nabla f(X)$ as $Y$ in \eqref{eq:Y_6terms} and by setting $\dot{W}_1$, $\tilde{W}_2$, and $\hat{W}_3$ in \eqref{eq:param_Y_TC} to zero because $\mathcal{P}_{T_X \mathbb{R}_{(r_1, r_2)}^{n_1 \times n_2 \times n_3}} \nabla f(X) = 0$, we obtain
\begin{equation*}
\begin{split}
&\nabla f(X)
=~ P_{X_1'}^\perp \cdot \nabla f(X) \cdot P_{X_3''^\top}^\perp \\
+ ~ & X_1' \cdot \Big[ P_{X_2'^{\mathrm{R}}}^\perp  \left( X_1'^\top \cdot \nabla f(X)  \right)^{\mathrm{R}}\Big]^{r_1 \times n_2 \times n_3} \cdot P_{X_3''^\top}^\perp \\
+ ~ & P_{X_1'}^\perp \cdot \Big[ \left(  \nabla f(X) \cdot X_3''^\top \right)^\mathrm{L}  P_{\left( X_2''^{\mathrm{L}}\right)^\top}^\perp \Big]^{n_1 \times n_2 \times r_2} \cdot X_3''.
\end{split}
\end{equation*}
Multiplying both sides on the right by $X_3''^\top$, we obtain
\begin{equation*}
\begin{split}
& \nabla f(X) \cdot X_3''^\top =\\
 &
~ P_{X_1'}^\perp \cdot \Big[ \left(  \nabla f(X) \cdot X_3''^\top \right)^\mathrm{L} P_{\left( X_2''^{\mathrm{L}}\right)^\top}^\perp \Big]^{n_1 \times n_2 \times r_2} = U_1 \cdot \dot{V}_2,
\end{split}
\end{equation*}
with $U_1$ and $\dot{V}_2$ as in \eqref{eq:param_Y_TC} with $Y$ replaced by $\nabla f(X)$.

Furthermore, from \Cref{lemma:local_min_best_approx}, we know that $X_1' = A_1'B_1$, $X_3'' = B_3 A_3''$, and $X_2=B_1^\top \cdot A_2 \cdot B_3^\top$, for some $B_1 \in \Stiefel(r_1,r_1')$ and $B_3^\top \in \Stiefel(r_2,r_2')$. Therefore,
\begin{equation*}
\begin{split}
&\rank \left( \left( \nabla f(X) \cdot X_3''^\top \right)^\mathrm{L} \right) \\
&= \rank \left( \left( X_1' \cdot X_2 - A_1' \cdot A_2 \cdot A_3'' \cdot X_3''^\top \right)^\mathrm{L} \right) \\
&= \rank \left( \left( A_1' B_1 \cdot B_1^\top \cdot A_2 \cdot B_3^\top  - A_1' \cdot A_2 \cdot B_3^\top \right)^\mathrm{L} \right) \\
&= \rank \left( A_1' \left( B_1 B_1^\top - I_{r_1'} \right) \left( A_2 \cdot B_3^\top \right) ^\mathrm{L} \right) \\
&= \mathrm{rank} \left( I_{r'_1} - B_1 B_1^\top \right) = r'_1 - r_1
\end{split}
\end{equation*}
because $\left( A_2 \cdot B_3^\top \right)^\mathrm{L}$ has full rank $r_1'$, knowing that $X_2^\mathrm{L} = B_1^\top \left(A_2 \cdot B_3^\top \right)^\mathrm{L}$ has rank $r_1$ and using the Sylvester rank inequality. 
Thus, $\rank \left( U_1 \dot{V}_2^{\mathrm{L}} \right) = r'_1 - r_1$.
A similar derivation can be made for $\left( X_1'^\top \cdot \nabla f(X) \right)^\mathrm{R} = U_2^\mathrm{R} V_3$. Hence, $\nabla f(X)$ can be parameterized as in \eqref{eq:paran_g_final} with $s_1= r'_1-r_1$, $s_2 = r'_2-r_2)$, and thus by definition $\nabla f(X) \in T_{X} \mathbb{R}_{\le (r_1', r_2')}^{n_1 \times n_2 \times n_3}$.
\end{proof}

We propose to exploit the two equalities from \Cref{prop:rank_est_best_low_rank_approx} in the context of LRTCP \eqref{eq:min_completion_TT} by using the \emph{estimated rank} of $B \in \mathbb{R}^{n \times m}$ which is inspired by \cite{gao2022riemannian} and defined as:
\begin{equation} \label{eq:num_rank_increase}
\begin{split}
&\tilde{r}_s \left( B \right) := \begin{cases} 
0 & \text{if } B = 0, \\
\argmax_{\substack{j\le s}} \frac{\sigma_j\left( B\right) - \sigma_{j+1}\left( B \right)}{\sigma_j \left( B\right)}& \text{otherwise},
\end{cases}
\end{split}
\end{equation}
where $\sigma_{j}\left( B \right)$, $j= 1 \dots \rank(B)$, denote the singular values of $B$ in decreasing order, i.e., $\sigma_i(B) \geq \sigma_j(B)$ for $i \leq j$, and $ s < \mathrm{rank} \left( B \right)$. The upper bound $s$ prevents the estimated rank from being too high and should be chosen by the user. 
Thus, we propose $\left(r_1 + \tilde{r}_s \big( \left( \nabla f_{\Omega}(X) \cdot X_3''^\top \right)^\mathrm{L} \big), r_2 + \tilde{r}_s \big( \left( X_1'^\top \cdot \nabla f_{\Omega}(X) \right)^R \big)\right)$ as an adequate value for $(k_1,k_2)$ in \eqref{eq:min_completion_TT}, where $X=X_1' \cdot X_2 \cdot X_3''$ has been obtained by running a Riemannian optimization algorithm on
\begin{equation} \label{eq:min_completion_TT_fixed}
\min_{X \in \mathbb{R}_{ (r_1, r_2)}^{n_1 \times n_2 \times n_3}} {\frac{1}{2} \lVert X_{\Omega} - A_{\Omega} \rVert^2}.
\end{equation} 

\section{Experiments} \label{sec:experiments}
In this section, three experiments are generated where we have optimized \eqref{eq:min_completion_TT_fixed} for $(r_1,r_2) := (2,2)$, $n_1:=n_2:=n_3:=100$, and $(r_1',r_2'):= (6,6)$, using a Riemannian conjugate gradient (CG) algorithm \cite{Steinl_high_dim_TT_compl_2016,manopt}. The tensor $A$ and the starting point $X_0$ given to the optimization algorithm are generated as follows:
\begin{equation*} \label{eq:A_randn}
\begin{split}
A = \mlin{randn} \left( n_1,r_1' \right) 
\cdot \mlin{randn} \left( r_1', n_2, r_2' \right) \cdot \mlin{randn} \left( r_2', n_3 \right), \\
X_0 = \mlin{randn} \left( n_1,r_1 \right) 
\cdot \mlin{randn} \left( r_1, n_2, r_2 \right) \cdot \mlin{randn} \left( r_2, n_3 \right),
\end{split}
\end{equation*}
where $\mlin{randn}$ is a built-in \Matlab function to generate pseudo-random numbers. It can be shown that the elements of $A$, generated in this way, have standard deviation $\sqrt{r'_1 r'_2}=6$. To obtain $A_\Omega$, $4 \cdot 10^4$ random samples of this tensor were generated.
An illustration of how the estimated rank of $\left( \nabla f_{\Omega }(X) \cdot X_3''^\top \right)^\mathrm{L}$ can be used to estimate a good value for $(k_1,k_2)$ in \eqref{eq:min_completion_TT} is given in \Cref{fig:svd_grad_Omega_vs_proj_rA6_6_local_opt_ni_100_Omega_40000}. The first $20$ singular values of $\left( \nabla f_{\Omega }(X) \cdot X_3''^\top \right)^\mathrm{L}$ are shown in the left upper subfigure. The squared norm of the Riemannian gradient that is obtained at $X$ is approximately $10^{-8}$. There were 200 iterations needed to obtain this accuracy. Based on the upper right subfigure, where the relative gap between the singular values is shown, it can be seen that the estimated rank equals $r'_1 - r_1$. In the lower two subfigures, the first $20$ singular values of $\nabla f_\Omega (X)^\mathrm{L}$ are shown. The estimated rank of $\nabla f_\Omega (X)^\mathrm{L}$ equals $3$, and thus cannot be used to estimate the rank of $A$. 

\begin{figure}
\centering
\includegraphics[width= \linewidth]{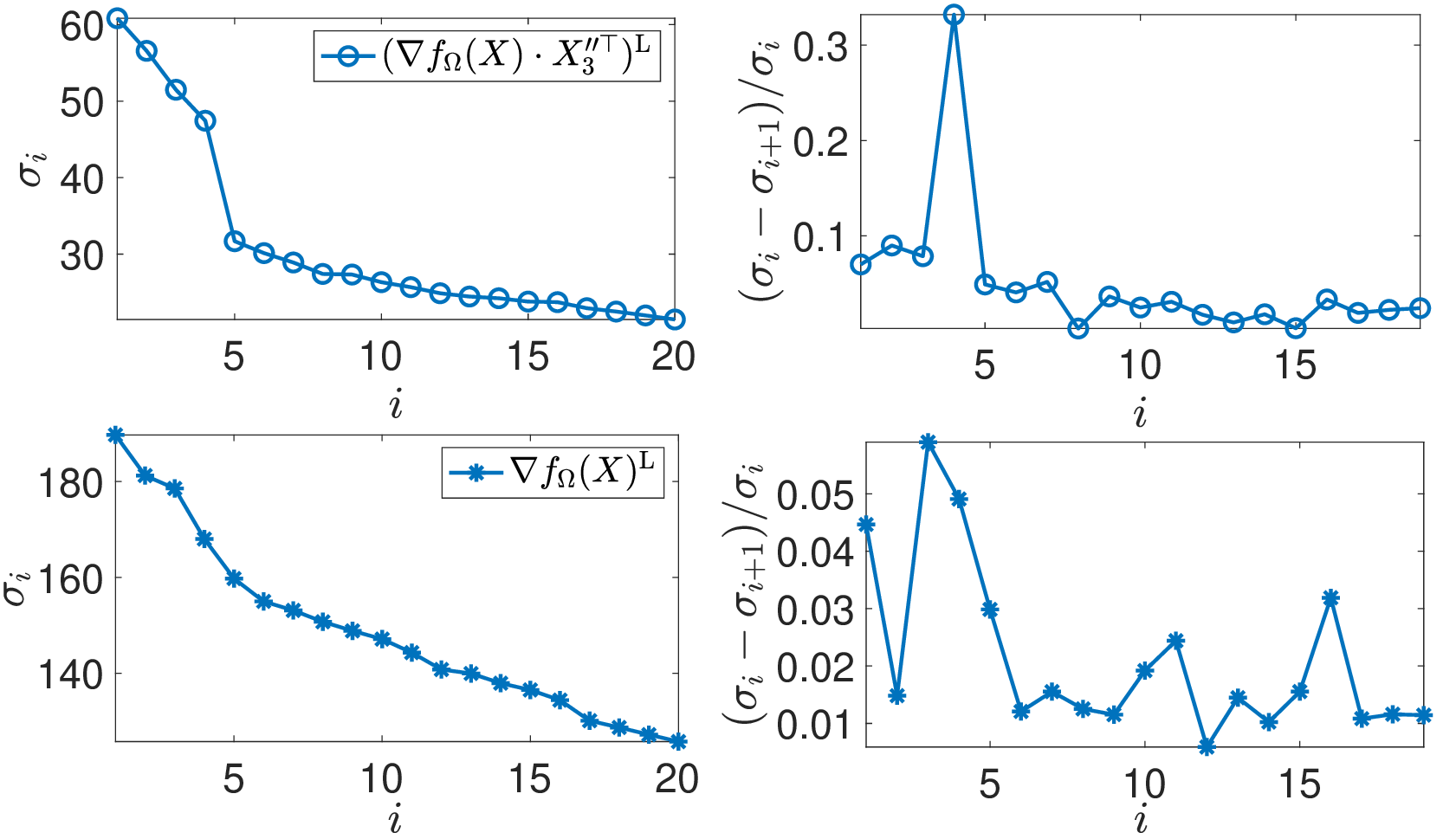}
\caption{An illustration of the advantage of $\tilde{r}_{20} \left( \left( \nabla f_{\Omega }(X) \cdot X_3''^\top \right)^\mathrm{L} \right)$ compared to $\tilde{r}_{20} \big( \nabla f_{\Omega}(X)^\mathrm{L}\big)$, to estimate the rank of $A$, for $\lVert \mathcal{P}_{T_X \mathbb{R}_{(r_1,r_2)}^{n_1 \times n_2 \times n_3}} \nabla f_{\Omega}(X) \rVert^2 = 10^{-8}$, obtained after 200 iterations. }
\label{fig:svd_grad_Omega_vs_proj_rA6_6_local_opt_ni_100_Omega_40000}
\end{figure}

In \Cref{fig:svd_grad_proj_rx2_n100_3D_r6_Omega_4_10+4_gradR2_10_60it}, it is shown that in practice the norm of the Riemannian gradient does not need to be very small for the estimated rank of $\left( \nabla f_{\Omega }(X) \cdot X_3''^\top \right)^\mathrm{L}$ to equal $r'_1 - r_1$. In this experiment, only 10 iterations of the Riemannian CG algorithm where used, such that the squared norm of the Riemannian gradient was approximately $504$. However the estimated rank still equals $r_1'-r_1$.

\begin{figure}
\centering
\includegraphics[width= \linewidth]{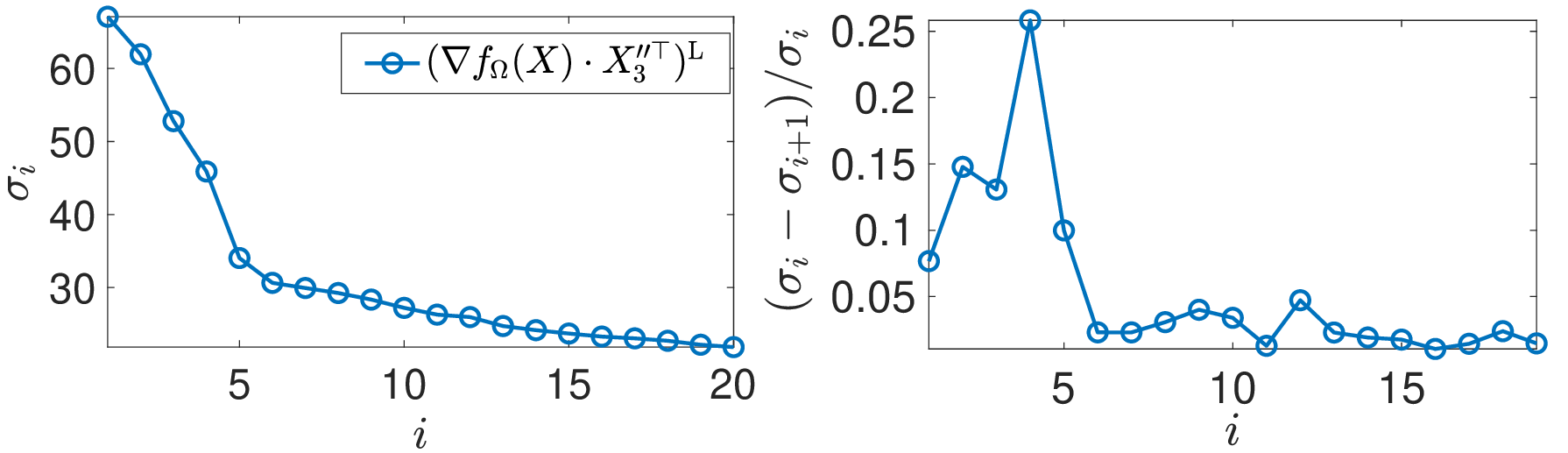}
\caption{The first 20 singular values of $\left( \nabla f_{\Omega }(X) \cdot X_3''^\top \right)^\mathrm{L}$ (left) and their relative gap (right), for $\lVert \mathcal{P}_{T_X \mathbb{R}_{(r_1,r_2)}^{n_1 \times n_2 \times n_3}} \nabla f_{\Omega}(X) \rVert^2 = 504$, obtained after 10 iterations.}
\label{fig:svd_grad_proj_rx2_n100_3D_r6_Omega_4_10+4_gradR2_10_60it}
\end{figure}

In a last experiment, another advantage of the proposed method is illustrated. For this experiment noise with $\eta = 10$ is added to the low-rank tensor as follows:
\begin{equation} \label{eq:A_noise}
\begin{split}
A_{\eta} = A + \eta \mlin{ randn} \left( n_1, n_2, n_3 \right).
\end{split}
\end{equation}
This means that the noise has the same magnitude as $A$ but the estimated rank of $\left( \nabla f_{\Omega }(X) \cdot X_3''^\top \right)^\mathrm{L}$ still equals $r_1'-r_1=4$ after 120 iterations of the CG algorithm, as shown in \Cref{fig:svd_grad_proj_rx2_n100_3D_r6_Omega_gradR2_100_20it_Omega_8_10+4_noise_10+0}. 
The squared norm of the Riemannian gradient equals 0.9.

\begin{figure}
\centering
\includegraphics[width= \linewidth]{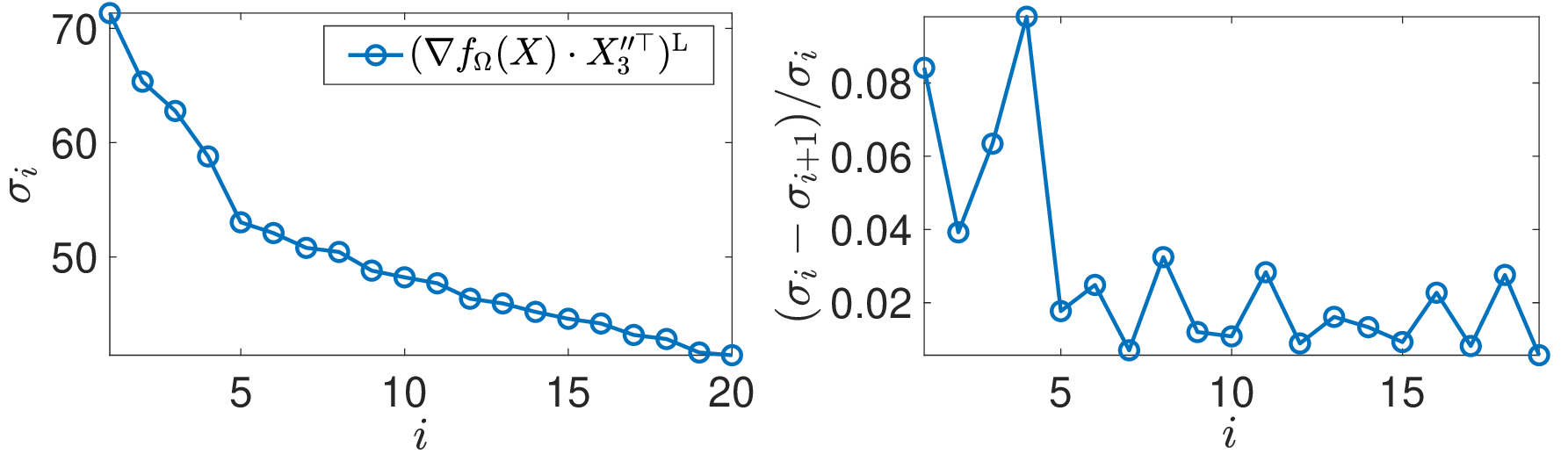}
\caption{The first 20 singular values of $\left( \nabla f_{\Omega }(X) \cdot X_3''^\top \right)^\mathrm{L}$ (left) and their relative gap (right), for $\lVert \mathcal{P}_{T_X \mathbb{R}_{(r_1,r_2)}^{n_1 \times n_2 \times n_3}} \nabla f_{\Omega}(X) \rVert^2 = 0.9$, obtained after 120 iterations, and with noise added to the data as in \eqref{eq:A_noise} with $\eta = 10 $.}
\label{fig:svd_grad_proj_rx2_n100_3D_r6_Omega_gradR2_100_20it_Omega_8_10+4_noise_10+0}
\end{figure}

\section{Conclusion}
The two equalities given in \Cref{prop:rank_est_best_low_rank_approx} enable to compute the TT-rank of $A$ based on a stationary point of LRTAP on the fixed-rank manifold, i.e., $\min_{X \in \mathbb{R}_{(r_1, r_2)}^{n_1 \times n_2 \times n_3}} \frac{1}{2}\|X-A\|^2$, which can be obtained using classic Riemannian optimization. Moreover, numerical experiments indicate that, for LRTCP \eqref{eq:min_completion_TT}, using these equalities with the rank replaced by the estimated rank \eqref{eq:num_rank_increase} provides a plausible estimation of the TT-rank of $A$ which can be used as an adequate value for $(k_1, k_2)$.

We are working on a Riemannian rank-adaptive method using this rank estimation method on the LRTCP and additionally on an extension of this method to higher dimensions. 

\section*{Acknowledgment}
This work was supported by the Fonds de la Recherche Scientifique -- FNRS and the Fonds Wetenschappelijk Onderzoek -- Vlaanderen under EOS Project no 30468160, and by the Fonds de la Recherche Scientifique -- FNRS under Grant no T.0001.23. 

\bibliographystyle{ieeetr}
\bibliography{references}

\end{document}